\documentclass{amsart}
\usepackage{amssymb}

\newtheorem{theorem}{Theorem}[section]
\newtheorem{lemma}[theorem]{Lemma}

\theoremstyle{definition}

\newtheorem{proposition}[theorem]{Proposition}
\newtheorem{example}[theorem]{Example}

\theoremstyle{remark}
\newtheorem{remark}[theorem]{Remark}

\numberwithin{equation}{section}

\begin{document}

\title{Equations defining recursive extensions as set theoretic complete intersections}

\author{Tran Hoai Ngoc Nhan}
\address{Department of Mathematics, Dong Thap University, Viet Nam}
\email{tranhoaingocnhan@gmail.com}

\author{Mesut \c{S}ah\.{i}n}
\address{Department of Mathematics, Hacettepe University, 06800, Ankara, ~ Turkey}
\email{mesutsahin@gmail.com}
\thanks{The second author is supported by T\"{U}B\.{I}TAK-2219}

\subjclass[2010]{14M10; 14M25;14H45.} \keywords{set-theoretic complete
intersections, monomial curves.}

\date{\today}

\begin{abstract}
Based on the fact that projective monomial curves in the plane are complete intersections, we give an effective inductive method for creating infinitely many monomial curves in the projective $n$-space that are set theoretic complete intersections. We illustrate our main result by giving different infinite families of examples. Our proof is constructive and provides one binomial and $(n-2)$ polynomial explicit equations for the hypersurfaces cutting out the curve in question.
\end{abstract}

\maketitle

\section{Introduction}

One of the most important and longstanding
open problems in classical algebraic geometry is to determine the least number of equations needed to define an
algebraic variety. This number which is also known as the arithmetical rank of the variety is bounded below by its codimension and above by the dimension of the ambient space, see \cite{eis}. Algebraic varieties whose arithmetical ranks coincide with their codimensions are called set theoretic complete intersections. Hence, an interesting problem is to ask if a given variety is a set theoretic complete intersection or not. Although there are algorithms for finding minimal generating sets for its ideal, there is no general theory for providing minimal explicit equations defining the variety set theoretically. Therefore a related and more challenging problem is to find codimension many polynomial equations which define a given set theoretic complete intersection. Finding explicit equations for parametrized curves also attracts attention for applications in geometric modeling (see e.g. \cite{goldman, juttler}).

 Let $K$ be an algebraically closed field of any
characteristic and $m_1<\ldots<m_n$ be some positive integers such
that $\textrm{gcd}(m_1,\ldots,m_n)= 1$. Recall that a monomial curve
$C(m_1,m_2,\ldots,m_n)$ in the projective space $\mathbb{P}^n$ over
$K$ is a curve with generic zero $(u^{m_n},\
u^{m_n-m_1}v^{m_1},\ldots,u^{m_n-m_{n-1}}v^{m_{n-1}},\ v^{m_n})$
where $u,v\in K$ and $(u,v)\neq (0,0)$. It
is known that every monomial curve in $\mathbb{P}^n$ is a
set-theoretic complete intersection, where $K$ is of characteristic
$p > 0$, see \cite{hart,moh,bmt-simplicial}. In the characteristic zero case,
there are partial results \cite{OS,SC,SCI} and efficient methods for finding new examples from old, see \cite{eto,kat,N,PS,OT,Tgluing} and the references therein for the current activity.

Even though a monomial curve in $n$-space is defined by either $n-1$ or $n$ equations set theoretically, these equations are given explicitly only in particular situations. Indeed, Moh provided $n-1$ binomial equations defining the curve in question set theoretically in positive characteristic, see \cite{moh}. In characteristic zero case, Thoma proved that this is possible, namely a monomial space curve is given by $2$ binomial equations, only if its ideal is generated by these binomials, see \cite{MSI}.  Three binomial equations cutting out a monomial curve in $\mathbb{P}^3$ is given by Barile and Morales in \cite{bm}. Later, Thoma generalized these by proving that every monomial curve in $n$-space is defined by $n$ binomial equations set theoretically and that $n-1$ binomial equations are sufficient if the curve is an ideal theoretic complete intersection, see \cite{bara}. He also discussed what type of equations would be needed if the monomial curve in $\mathbb{P}^3$ was given by $2$ equations, see \cite{OTE}. Eto, on the other hand, studied in \cite{eto} monomial curves defined by $n-2$ binomials plus a polynomial.

The aim of this paper is to use the fact that monomial plane curves are complete intersections and give an elemantary proof of the fact (due to Thoma \cite{OT}) that their recursive extensions are set theoretic complete intersections under a mild condition. Our main contribution here is to give one binomial and $(n-2)$ non-binomial explicit equations for the hypersurfaces cutting out the curves in question. Our main technique is a combination of the methods of \cite{hart,N} and of \cite{tjm,PS}.

\section{The Main Result}
In this section, we prove our main theorem, which can be used to construct infinitely
many set-theoretic complete intersection monomial curves in
$\mathbb{P}^n$. Throughout the paper, we study monomial curves $C(m_1,m_2,\ldots,m_n)$ in $\mathbb{P}^n$, where $m_i
\in \langle m_1,m_2,\dots,m_{i-1} \rangle$ for every $3\leq i\leq n$, so that $m_i
=a_{i,1}m_1+\cdots+a_{i,i-1}m_{i-1}$ for some nonnegative
integers $a_{i,j}$. Note that each monomial curve $C(m_1,m_2,\ldots,m_i)$ in $\mathbb{P}^i$ is an extension of $C(m_1,m_2,\ldots,m_{i-1})$ in $\mathbb{P}^{i-1}$, for every $3\leq i\leq n$, in the language of \cite{PS,ext}. From now on,  $C \subseteq \mathbb{P}^n$ denotes a monomial curve $C(m_1,\dots,m_n)$ of this form and is referred to as a \textit{recursive extension}. Here is the first observation about these special curves.
\begin{lemma}\label{lem1} If $C \subseteq \mathbb{P}^n$ is a recursive extension and $3\leq i\leq n$, then the following equivalent conditions hold. \\
{\rm (I)} There are non-negative integers $a_i$, $b_i$ and $c_i$ satisfying\\ $m_i=c_im_{i-1}+b_im_2+a_im_1$, together with $m_{i-1}>a_i$, $m_{i-1}>b_i$, and $b_3=0$.\\
{\rm (II)} There are non-negative integers $\alpha_i$, $\beta_i$ and $\gamma_i$ satisfying\\ $m_i
=\gamma_im_{i-1}-\beta_im_2-\alpha_im_1$ together with $m_{i-1}>\alpha_i$, $m_{i-1}>\beta_i$, $\beta_3=0$.
\end{lemma}

\begin{proof} In this case, we can easily write $m_i=C_im_{i-1}+B_im_2+A_im_1$ for some nonnegative integers $A_i$,
$B_i$ and $C_i$,  for all $3\leq i\leq n$, where $B_3=0$. Denoting by $\lfloor a \rfloor$ the largest integer less than or equal to
$a$ and setting $a_i=A_i-\lfloor\frac{A_i}{m_{i-1}}\rfloor m_{i-1}, b_i=B_i-\lfloor\frac{B_i}{m_{i-1}}\rfloor m_{i-1}$ and $c_i=C_i+\lfloor\frac{B_i}{m_{i-1}}\rfloor m_2+\lfloor\frac{A_i}{m_{i-1}}\rfloor m_1$, we can further express these integers, for all
$3\leq i\leq n$, as $$m_i=c_im_{i-1}+b_im_2+a_im_1$$ so that
$m_{i-1}>a_i\geq 0$, $m_{i-1}>b_i\geq 0$, $b_3=0$ and $c_i\geq 0$. This proves the first part. As for the second part and for the equivalence, one can use the following formulas:
$$\alpha_i=\left\{
                       \begin{array}{ccc}
                         0 & \text{if} & a_i=0 \\
                         m_{i-1}-a_i & \text{if} & a_i >0,
                       \end{array}
                     \right.
\beta_i=\left\{
                       \begin{array}{ccc}
                         0 & \text{if} & b_i=0 \\
                         m_{i-1}-b_i & \text{if} & b_i >0,
                       \end{array}
                     \right.$$
and $$\gamma_i=\left\{
                       \begin{array}{ccc}
                         c_i & \text{if} & a_i=0 \quad \text{and}\quad  b_i =0,\\
                         c_i+m_2 & \text{if} & a_i=0 \quad \text{and}\quad  b_i >0,\\
                         c_i+m_1 & \text{if} &  a_i >0 \quad \text{and}\quad b_i=0, \\
                          c_i+m_2+m_1 & \text{if} &   a_i >0 \quad \text{and}\quad  b_i >0.
                       \end{array}
                     \right.$$
\end{proof}

\begin{example} Consider the monomial curve $C(1,2,3,5)\subset \mathbb{P}^4$. Then, the integers in Lemma \ref{lem1} are not unique as can be seen below.
\begin{align*}5&=1\cdot3+1\cdot2+0\cdot1=3\cdot3-2\cdot2-0\cdot1\\
&=1\cdot3+0\cdot2+2\cdot1=2\cdot3-0\cdot2-1\cdot1    \hfill {\qed}
\end{align*}
\end{example}

 The following is crutial to prove our main result.

\begin{lemma}\label{lem2} Let $C$ in
$\mathbb{P}^n$ be a recursive extension and $\alpha_i$, $\beta_i$ and $\gamma_i$ are some non-negative integers as in Lemma \ref{lem1}. If $\gamma_i-\beta_i-\alpha_i-1\geq 0$ and $m_1\geq\beta_i (m_2-m_1)$, then $F_{i-1}=x_{i-1}^{m_i}+G_{i-1}+H_{i-1}\in I(C)$ for all $3\leq i\leq n$, where
$$G_{i-1}=\sum_{k=1}^{N_i}(-1)^k {m_{i-1}\choose k}
x_1^{k\alpha_i}x_2^{k\beta_i}x_{i-1}^{m_i-k\gamma_i}
x_i^kx_0^{k(\gamma_i-\beta_i-\alpha_i-1)},$$

$$H_{i-1}=\sum_{k=N_i+1}^{m_{i-1}}(-1)^k{m_{i-1}\choose k}
x_1^{a_i(m_{i-1}-k)}x_2^{b_i(m_{i-1}-k)}x_{i-1}^{c_i(m_{i-1}-k)}x_i^k  x_0^{h_0},$$
$N_i=m_{i-1}-m_1$ and $ h_0=k(a_i+b_i+c_i-1)-b_i(m_{i-1}-m_2)-a_i(m_{i-1}-m_1).$

\end{lemma}

\begin{proof} First we prove that $G_{i-1}$ and $H_{i-1}$ are polynomials, i.e. their monomials have non-negative exponents.

For $G_{i-1}$, we only need to check the exponent of $x_{i-1}$. By $k \leq  N_i$, we have $m_i-k\gamma_i \geq m_i- (m_{i-1}-m_1)\gamma_i$. Since  $m_i=\gamma_im_{i-1}-\beta_im_2-\alpha_im_1$ and $\gamma_i  -\alpha_i \geq \beta_i+1$  it follows that $$m_i-k\gamma_i \geq (\gamma_i  -\alpha_i)m_1- \beta_i m_{2} \geq (\beta_i+1)m_1- \beta_i m_{2}=m_1-\beta_i(m_2-m_1).$$
Therefore, $m_i-k\gamma_i \geq 0$ by the hypothesis $m_1\geq\beta_i (m_2-m_1)$.

For $H_{i-1}$, we only check if $h_0 \geq 0$. Since $k\geq N_i+1=m_{i-1}-m_1+1$, $$h_0\geq (m_{i-1}-m_1+1)(c_i+b_i+a_i-1)-b_i(m_{i-1}-m_2+1)-a_i(m_{i-1}-m_1).$$ Thus,
 \begin{equation}\label{eqn1} h_0 \geq (m_{i-1}-m_1+1)(c_i-1)+b_i(m_{2}-m_1+1)+a_i.
\end{equation}
It follows that, $h_0 \geq 0$ as long as $c_i >0$. We now treat the case where $c_i=0$, in which case \eqref{eqn1} becomes
\begin{equation}\label{eqn2} h_0 \geq b_i(m_{2}-m_1+1)+a_i-(m_{i-1}-m_1+1).
\end{equation} Notice first that the assumption $\gamma_i-\beta_i-\alpha_i-1\geq 0$ yields
\begin{equation}\label{eqn3}  b_i-1 \geq (m_{i-1}-m_2)  \quad \text{if} \quad  a_i=0 \quad \text{and}\quad  b_i >0,
\end{equation}
\begin{align*} a_i-1 \geq (m_{i-1}-m_1) \quad &\text{if} \quad  a_i >0 \quad \text{and}\quad b_i=0, \\
 a_i+b_i-1 \geq (m_{i-1}-m_2)+(m_{i-1}-m_1) \quad &\text{if}   \quad a_i >0 \quad \text{and}\quad  b_i >0.
\end{align*}

We see immediately that $h_0\geq 0$ as soon as $a_i>0$. If $a_i=0$, then from \eqref{eqn2} and \eqref{eqn3}, we obtain
 \begin{eqnarray}\label{eqn4}  h_0 &\geq& (m_{i-1}-m_2+2)(m_{2}-m_1+1)-(m_{i-1}-m_2)\\
 \nonumber &>&(m_{2}-m_1+1)(m_{i-1}-m_2)\geq0.
\end{eqnarray}

To accomplish the goal of proving $F_{i-1}\in I(C)$, we make use the fact that $I(C)$ is the kernel of the surjective map defined by
$$\phi: K[x_0,\dots,x_n]\rightarrow K[u^{m_n},\
u^{m_n-m_1}v^{m_1},\ldots,u^{m_n-m_{n-1}}v^{m_{n-1}},\ v^{m_n}],$$ where $\phi(x_i)=u^{m_n-m_i}v^{m_i}$, for $i=0,\dots,n$ with the convention $m_0=0$. Recall that $F\in I(C)=\ker(\phi)$ iff the sum of the coefficients of $F$ is zero and $F$ is homogeneous with respect to the grading afforded by $deg_C(x_i)=(m_n-m_i,m_i)$. It is not difficult to check that the monomials in $F_{i-1}$ have degree $m_i(m_n-m_{i-1},m_{i-1})$ and thus $F_{i-1}$ is homogeneous with respect to this grading. Since $\sum_{k=0}^{m_{i-1}}(-1)^k {m_{i-1}\choose k}=0$, the proof is complete.
\end{proof}

\begin{theorem}\label{dl1} Let $C$ in
$\mathbb{P}^n$ be a recursive extension and $\alpha_i$, $\beta_i$ and $\gamma_i$ are some non-negative integers as in Lemma \ref{lem1}. If $\gamma_i-\beta_i-\alpha_i-1\geq 0$ and $m_1\geq\beta_i (m_2-m_1)$, then $C$ is a set-theoretic complete intersection on $F_1= x_1^{m_2}-x_2^{m_1}x_0^{m_2-m_1}$ and $F_2,\dots,F_{n-1}$ defined in Lemma \ref{lem2}.
\end{theorem}
\begin{proof} It is clear that $F_1 \in I(C)$. Together with Lemma \ref{lem2}, this reveals that $C$ lies on the hypersurfaces defined by these polynomials. Therefore, it is sufficient to prove that the common zeroes of the system $F_1 =\dots =F_{n-1}=0$ lies in $C$.

If $x_0 = 0$, $F_1=0$ yields $x_1 = 0$, and in this case
we first prove that $x_2 =\dots =x_{n-1}=0$ by $F_2 =\dots =F_{n-1}=0$. It follows easily from \eqref{eqn1} that $h_0 >0$, when $c_i \geq 1$, as otherwise we would get $c_i=1$, $b_i=a_i=0$ and thus $m_i=m_{i-1}$. Assume now that $c_i=0$. If $a_i=0$ and $b_i>0$, then $h_0>0$ by \eqref{eqn4}. If $a_i>0$ and $b_i>0$, $a_i+b_i-1 > (m_{i-1}-m_1)$ by \eqref{eqn3} and hence $h_0>0$. If $a_i>0$ and $b_i=0$, $h_0>0$ except when $a_i=k=N_i+1=m_{i-1}-m_1+1$ and $m_1 >1$ in which case $x_1$ divides the monomial in $H_{i-1}$ corresponding to $k=N_i+1$. In any case, $H_{i-1}=0$ whenever $F_{i-1}=0$ and $x_0=x_1=0$. Similarly, $G_{i-1}=0$ whenever $F_{i-1}=0$ and $x_0=x_1=0$, under the assumption $\alpha>0$ or $\gamma_i-\beta_i-\alpha_i-1>0$. So, assume that $\alpha_i=0$ and $\gamma_i-\beta_i-\alpha_i-1=0$ for every $3\leq i\leq n$. Then, $\beta >0$ for $4\leq i\leq n$ as otherwise we would get $\gamma_i=1$ and $m_i=m_{i-1}$. Since $\beta_3=0$, $\alpha_3=0$ and $\gamma_3-\beta_3-\alpha_3-1=0$ can not occur. So, $G_2=0$ whenever $F_{2}=0$ and $x_0=x_1=0$, which together with $H_2=0$ implies that $x_2=0$. On the other hand, $x_2$ divides $G_{i-1}$ when $\beta_i>0$, for every $4\leq i\leq n$. Therefore, in any case $G_{i-1}=0$ whenever $F_{i-1}=0$ and $x_0=x_1=0$, for every $4\leq i\leq n$. These prove our claim that $x_2 =\dots =x_{n-1}=0$ by $F_2 =\dots =F_{n-1}=0$. Thus, the common solution is just the point $(0,\ldots,0,1)$ which
lies on $C$.

On the other hand, we can
set $x_0=1$ when $x_0 \neq 0$. Therefore, it is sufficient to show
that the only common solution of these equations is $x_i=t^{m_i}$,
for some $t \in K$ and for all $1\leq i\leq n$, which we prove by
induction on $i$. More precisely, we show that if
$F_{i-1}(1,x_1,\dots,x_n)=0$, and $x_1 = t^{m_1}$, \ldots,
$x_{i-1} = t^{m_{i-1}}$ then $x_i=t^{m_i}$, for all $2\leq i\leq n$.
Clearly $m_i=c_im_{i-1}+b_im_2+a_im_1$ implies
$\textrm{gcd}(m_1,\ldots,m_{i-1})=1$ for all $3\leq i\leq n$. In
particular, $\textrm{gcd}(m_1,m_2)=1$, which means that there are
integers $\ell_1,\ell_2$ such that $\ell_1$ is positive and
$\ell_1m_2+\ell_2m_1=1$. From the first equation $F_1=0$, $x_1^{m_2}
= x_2^{m_1}$. Letting $x_1=T^{m_1}$, we get $x_2= \varepsilon
T^{m_2}$, where $\varepsilon$ is an $m_1$-st root of unity. Setting
$t=\varepsilon^{\ell_1} T$, we get $x_1 = t^{m_1}$ and
$x_2=t^{m_2}$, which completes the base statement for the induction.

Now, we assume that $x_0 = 1,\ x_1 = t^{m_1},\ldots, x_{i-1} =
t^{m_{i-1}}$ for some $3\leq i\leq n$. Substituting these to the
equation $F_{i-1}=0$, we get\\
$\displaystyle 0=(t^{m_{i-1}})^{m_i}+\sum_{k=1}^{N_i}(-1)^k
{m_{i-1}\choose k}
(t^{m_1})^{k\alpha_i}(t^{m_2})^{k\beta_i}(t^{m_{i-1}})^{m_i-k\gamma_i}x_i^k$

$$+ \sum_{k=N_i+1}^{m_{i-1}}(-1)^k
{m_{i-1}\choose k}
(t^{m_1})^{a_i(m_{i-1}-k)}(t^{m_2})^{b_i(m_{i-1}-k)}(t^{m_{i-1}})^{c_i(m_{i-1}-k)}x_i^k.$$
Since $m_i=\gamma_im_{i-1}-\beta_im_{2}-\alpha_im_1=c_im_{i-1}+b_im_2+a_im_1$,
this is nothing but
$$\overset{m_{i-1}}{\underset{k=0}{\sum}}(-1)^k
                                       \left(
                                          \begin{array}{c}
                                            m_{i-1} \\
                                            k \\
                                          \end{array}
                                        \right)
                                           (t^{m_i})^{m_{i-1}-k}x_i^{k}=(t^{m_i}-x_i)^{m_{i-1}}=0.$$
Hence $x_i=t^{m_i}$ completing the proof.
\end{proof}

\begin{example} Consider the monomial curve  $C(1,2,3,m_4)\subset \mathbb{P}^4$, where $m_4>3$. Note that $m_1\geq\beta_i(m_2-m_1)$ is satisfied if and only if $\beta_i \in \{0,1\}$. Clearly, we have $3=1\cdot2+1\cdot1=2\cdot2-1\cdot1$ so $\beta_3=0$ and $\gamma_3=\beta_3+\alpha_3+1$. By Theorem \ref{dl1}, the rational normal curve $C(1,2,3)\subset \mathbb{P}^3$ is a set theoretic complete intersection on $F_1, F_2$, where
 \begin{eqnarray}
  \nonumber      F_1&=&x_1^2-x_0x_2, \\
  \nonumber      F_2&=&x_2^3-2x_1x_2x_3+x_3^2x_0.
 \end{eqnarray}
Either $m_4=3c_4$, for some integer $c_4>1$; or for some positive integer $c_4$, we have $m_4=3c_4+1$ or $m_4=3c_4+2$. If $m_4=3c_4$, then $m_4=c_4\cdot3-0\cdot2-0\cdot1$, so $\beta_4=0$ and $\gamma_4\geq\beta_4 + \alpha_4+1 $ for every $c_4>1$. So, $C(1,2,3,m_4)\subset \mathbb{P}^4$ is a set theoretic complete intersection on $F_1, F_2,F_3$, where $F_3$ is as follows
$$x_3^{3c_4}-3x_3^{2c_4}x_4x_0^{c_4-1}+3x_3^{c_4}x_4^2x_0^{2c_4-2}-x_4^3x_0^{3c_4-3}, ~ \text{if}~ m_4=3c_4 \geq6.$$
 When $m_4=3c_4+1$, then $m_4=(c_4+1)\cdot3-0\cdot2-2\cdot1$, so $\beta_4=0$ and the condition $\gamma_4\geq\beta_4 + \alpha_4+1 $ is satisfied for every $c_4>0$ except $c_4=1$. So, $C(1,2,3,m_4)\subset \mathbb{P}^4$ is a set theoretic complete intersection on $F_1, F_2,F_3$, where $F_3$ is as follows
$$x_3^{3c_4+1}-3x_1^2x_3^{2c_4}x_4x_0^{c_4-2}+3x_1^4x_3^{c_4-1}x_4^2x_0^{2c_4-4}-x_4^3x_0^{3c_4-2}, ~\text{if} ~m_4=3c_4+1 \geq7.$$
For the exception $m_4=4$, we see that $4=0\cdot 3+2\cdot2+0\cdot1=2\cdot 3-1\cdot2-0\cdot1$ so $\beta_4=1$ and $\gamma_4-\beta_4-\alpha_4-1=0.$
Hence, $C(1,2,3,4)\subset \mathbb{P}^4$ is a set theoretic complete intersection on $F_1, F_2,F_3$, where
 $$ F_3=x_3^4-3x_2x_3^2x_4+3x_2^2x_4^2-x_4^3x_0.$$
Finally, if $m_4=3c_4+2$, then $m_4=(c_4+2)\cdot3-2\cdot2-0\cdot1$ and the condition $\gamma_4\geq\beta_4 + \alpha_4+1 $ is satisfied for every $c_4>0$ but $\beta_4=2$ meaning that $m_1\geq\beta_4(m_2-m_1)$ is not satisfied. The latter condition was just to make sure that the power of $x_{i-1}$ in $G_{i-1}$ is non-negative. Since,
$$F_3=x_3^{3c_4+2}-3x_2^2x_3^{2c_4}x_4x_0^{c_4-1}+3x_2^4x_3^{c_4-2}x_4^2x_0^{2c_4-2}-x_4^3x_0^{3c_4-1}$$ and the powers of $x_3$ are non-negative for every $c_4\geq 2$ and $F_3$ is clearly a polynomial. Thus, Theorem \ref{dl1} still applies and $C(1,2,3,4)\subset \mathbb{P}^4$ is a set theoretic complete intersection on $F_1, F_2,F_3$ if $m_4=3c_4+2 \geq8.$
For the exceptional case where $m_4=5$, we have $5=3\cdot 3-2\cdot2-0\cdot1=2\cdot 3-0\cdot2-1\cdot1$ so in both cases $\gamma_4-\beta_4-\alpha_4-1=0.$ But, in the first case $\beta_4=2$ and in the second case $\beta_4=0$ and we can apply Theorem \ref{dl1} with the second presentation. Therefore, $C(1,2,3,5)\subset \mathbb{P}^4$ is a set theoretic complete intersection on $F_1, F_2,F_3$, where
 $$ F_3=x_3^5-3x_1x_3^3x_4+3x_1^2x_3x_4^2-x_4^3x_0^2.$$

\end{example}

\begin{remark} In \cite{eto}, Eto studies necessary and sufficient conditions under which a monomial curve is a set theoretic complete intersection on $n-2$ binomials and one polynomial. In contrast, our curves are set theoretic complete intersections on one binomial $F_1$ and $n-2$ polynomials $F_2,\dots,F_{n-1}$ with more than two monomials.
\end{remark}

\begin{remark} Only when $\beta_i=0$ and $m_1=1$, Theorem \ref{dl1} is a special case of Theorem $2.1$ in \cite{N} but as long as $\beta_i>0$ or $m_1>1$ it improves upon the condition that $m_i$ must satisfy, for $i=3,\dots,n$. Namely, Theorem $2.1$ in \cite{N} requires for $m_i=\gamma_im_{i-1}-\beta_im_{2}-\alpha_im_1$ that $\gamma_i\geq\beta_i m_2+ \alpha_i m_1$ if $m_1>1$  and $\gamma_i\geq\beta_im_2 + \alpha_i+1 $ if $m_1=1$ whereas our main result needs only $\gamma_i\geq\beta_i + \alpha_i+1 $. It is an improvement also of Theorem $5.8$ in \cite{PS} in that starting from a monomial curve in $\mathbb{P}^2$ our main result can produce infinitely many new examples in $\mathbb{P}^n$ for every $n\geq3$ whereas Theorem $5.8$ in \cite{PS} can only produce them for $n=3$. Finally, Theorem 3.4 in \cite{OT} implies Theorem \ref{dl1} but its proof is not as elemantary as our proof and does not give the equations cutting out the curves.
\end{remark}

The following consequence, which illustrates the strength of our main theorem, can be proved by imitating the proof of Proposition $2.4$ in \cite{N}.
\begin{proposition}\label{hq1}
If $\textrm{gcd}(m_{i-1},m_1)=1$ and $m_i \geq
\textrm{max}\{m_{i-1}m_1,\ m_{i-1}(m_{i-1}-m_1)\}$, for all $3\leq i \leq n$, then the monomial curve $C(m_1,m_2,\dots,m_n)$ in
$\mathbb{P}^n$ is a set-theoretic complete intersection.
\end{proposition}
\begin{proof}
From the condition $\textrm{gcd}(m_{i-1},m_1) = 1$, there exist positive
integers $A_i$ and $B_i$ such that $m_i= A_im_{i-1}-B_im_1$. Since $m_i\geq m_{i-1}(m_{i-1}-m_1)$, we have $A_im_{i-1}-B_im_1 \geq
m_{i-1}(m_{i-1}-m_1)$. Subtracting $B_im_{i-1}$ from both hand sides
and rearranging, we obtain the following $$A_im_{i-1}-B_im_{i-1} \geq
-B_im_{i-1}+B_im_1+m_{i-1}(m_{i-1}-m_1).$$ Dividing both hand sides
by $m_{i-1}(m_{i-1}-m_1)$ yields
$$\dfrac{A_i-B_i}{m_{i-1}-m_1}\geq-\dfrac{B_i}{m_{i-1}}+1.$$ On the
other hand, the hypothesis $m_i\geq m_{i-1}m_1$ yields
$$\dfrac{A_i}{m_{1}}-\dfrac{B_i}{m_{i-1}}=\dfrac{m_i}{m_{i-1}m_{1}}\geq 1,\quad \mbox{which ~~implies}~~-\dfrac{B_i}{m_{i-1}}\geq\dfrac{m_1-A_i}{m_1}.$$  Therefore, we can
choose positive integers $\theta_i$ satisfying the condition
$$\dfrac{A_i-B_i}{m_{i-1}-m_1}\geq\dfrac{-B_i}{m_{i-1}}+1>\theta_i \geq-\dfrac{B_i}{m_{i-1}}\geq\dfrac{m_1-A_i}{m_1}.$$
Then we can set $\gamma_i= A_i+m_1\theta_i$ and $\alpha_i=
B_i+m_{i-1}\theta_i$ so that $$m_i = \gamma_i m_{i-1} - \alpha_i m_1 ~\mbox{where}~m_{i-1}>\alpha_i \geq 0, \gamma_i\geq m_1, \gamma_i-\alpha_i-1\geq 0.$$
Since $\beta_i =0$, the condition $m_1 \geq \beta_i(m_2-m_1)$ holds and it follows from Theorem \ref{dl1} that the
monomial curve $C(m_1,m_2,\ldots,m_n)$ is a set-theoretic complete
intersection.
\end{proof}

\section{Finding the Equations} In this section, we briefly explain how we find the equations cutting out the set theoretic complete intersections. We work within the most general set up but explain how we construct the polynomial $F_n$ for a fixed $n\geq 2$. Assume that $$ m_{n+1}=\beta m_{n}-\sum_{i=1}^{n-1} \alpha_i  m_i=\sum_{i=1}^{n} a_i m_i$$ for some non-negative $a_i$ and $\alpha_i$. These give us two homogeneous binomials:

$$(x_{n}^{\beta}-x_0^{\Delta}x_1^{\alpha_1} \cdots x_{n-1}^{\alpha_{n-1}}x_{n+1})^{m_n}$$   and

$$(x_1^{a_1} \cdots x_n^{a_n} -x_0^{\delta} x_{n+1})^{m_n}.$$

As in the proof of Theorem \ref{dl1} when we substitute $x_0 = 1,\ x_1 = t^{m_1},\ldots, x_{n} =t^{m_{n}}$, in our equation $F_{n}=0$, we would like to end up with $(t^{m_{n+1}}-x_{n+1})^{m_{n}}=0$. If we do the substituation in the first binomial we get $(t^{\beta m_{n}}-t^{\sum_{i=1}^{n-1}  \alpha_i m_i}x_{n+1})^{m_{n}}=0$ instead. To resolve this we divide the first polynomial by $x_n^{\sum_{i=1}^{n-1}  \alpha_i m_i}$ and to get the same degree in the monomials of both expressions we divide the second binomial by $x_0^{\sum_{i=1}^{n-1}a_i(m_n-m_i)}$. But some monomials will have negative powers and these two expressions are not polynomials anymore. If there exist an integer $N$ with $1<N<m_n$ such that $m_{n+1} \geq \beta k$ for $1\leq k \leq N$ and $m_{n+1} \geq k+ \sum_{i=1}^{n}a_i(m_n-k)$ for $N+1 \leq k \leq m_n$, then we can make up a polynomial $F_n$ by taking the first $N+1$ monomials from the first expression and by taking the rest from the second one. This explains why we restrict ourself in the main theorem. Let us illustrate this by an example:

Consider the rational normal curve $C=C(1,2,\dots,n,n+1) \subseteq \mathbb{P}^{n+1}$. We have $$n+1=2 \cdot n-1\cdot (n-1)=1\cdot n+1\cdot 1.$$
These give us the following binomials:
$$(x_{n}^{2}-x_{n-1}x_{n+1})^{n}=x_n^{2n}+\sum_{k=1}^{n}(-1)^k  \left(
                                          \begin{array}{c}
                                            n \\
                                            k \\
                                          \end{array}
                                        \right) x_n^{2n-2k}x_{n-1}^kx_{n+1}^k \quad \mbox{and}$$
$$(x_1x_{n}-x_{0}x_{n+1})^{n}=\sum_{k=0}^{n-1}(-1)^k  \left(
                                          \begin{array}{c}
                                            n \\
                                            k \\
                                          \end{array}
                                        \right) x_1^{n-k}x_n^{n-k}x_{0}^kx_{n+1}^k +(-1)^nx_0^{n}x_{n+1}^{n}.$$
Dividing the first one by $x_n^{n-1}$ and the second one by $x_0^{n-1}$, we get the following:
$$x_n^{n+1}+\sum_{k=1}^{n}(-1)^k  \left(
                                          \begin{array}{c}
                                            n \\
                                            k \\
                                          \end{array}
                                        \right) x_n^{n+1-2k}x_{n-1}^kx_{n+1}^k \quad \mbox{and}$$
$$\sum_{k=0}^{n-1}(-1)^k  \left(
                                          \begin{array}{c}
                                            n \\
                                            k \\
                                          \end{array}
                                        \right) x_1^{n-k}x_n^{n-k}x_{0}^{k-n+1}x_{n+1}^k +(-1)^nx_0x_{n+1}^{n}.$$
It is now clear that $x_n^{n+1-2k}$ is no longer a monomial for $k$ satisfying $2k>n+1$ in the first expression and $x_{0}^{k-n+1}$ defines a monomial only for the last two terms in the second expression. Thus, if we take $N=n-2$ and replace the last two terms of the first expression with the last two monomials, we get the following expression:
$$F_{n}=x_n^{n+1}+\sum_{k=1}^{N}(-1)^k  \left(
                                          \begin{array}{c}
                                            n \\
                                            k \\
                                          \end{array}
                                        \right) x_n^{n+1-2k}x_{n-1}^kx_{n+1}^k+(-1)^{n-1}n x_1x_nx_{n+1}^{n-1} +(-1)^nx_0x_{n+1}^{n}.$$
Note that this is a polynomial if and only if $n+1-2k \geq 0$ for all $1\leq k \leq N=n-2$, which holds if and only if $n\leq5$.

\bibliographystyle{amsplain}

\end{document}